\documentclass[smallextended,referee,envcountsect,]{svjour3}
\smartqed
\usepackage{graphicx}
\journalname{JOTA}
\usepackage{hyperref}
\usepackage{amssymb}
\usepackage{mathtools}

\usepackage{pgf,tikz}
\usepackage{algorithm}

\DeclareMathOperator*{\st}{s.t.}

\DeclareMathOperator*{\co}{co}

\newcommand{\vertiii}[1]{{\left\vert\kern-0.25ex\left\vert\kern-0.25ex\left\vert #1
    \right\vert\kern-0.25ex\right\vert\kern-0.25ex\right\vert}}
\usepackage{enumerate}
\journalname{}

\begin{document}

\title{On the rate of convergence of the Difference-of-Convex Algorithm (DCA)}

\author{Hadi Abbaszadehpeivasti   \and  Etienne de Klerk \and Moslem Zamani}

\institute{Hadi Abbaszadehpeivasti \at
              h.abbaszadehpeivasti@tilburguniversity.edu
           \and
           Etienne de Klerk, Corresponding author  \at
              e.deklerk@tilburguniversity.edu
           \and
           Moslem Zamani,    \at
              m.zamani\_1@tilburguniversity.edu
            \and
           Tilburg University, Department of Econometrics and Operations Research\\
              Tilburg, The Netherlands
}

\date{Received: date / Accepted: date}

\maketitle

\begin{abstract}
In this paper, we study the non-asymptotic convergence rate of the DCA (Difference-of-Convex Algorithm), also known as the convex-concave procedure,
with two different termination criteria that are suitable for smooth and nonsmooth decompositions respectively.
The DCA is a popular algorithm for difference-of-convex (DC) problems, and known to
converge to a stationary point of the objective under some assumptions.
 We derive a worst-case convergence rate of $O(1/\sqrt{N})$ after $N$ iterations of the objective gradient norm  for certain classes of  DC problems,
  without assuming strong convexity in the DC decomposition, and give an example which shows the convergence rate is exact. We also provide a new convergence rate of $O(1/N)$ for the DCA
  with the second termination criterion.
Moreover, we derive a new linear convergence rate result for the DCA under the assumption of the Polyak-\L ojasiewicz inequality.
The novel aspect of our analysis is that it employs semidefinite programming performance estimation.
\end{abstract}


\section{Introduction}
In this paper, we consider the general difference-of-convex (DC) optimization problem,
 \begin{align}    \label{P}
        \inf \  &f(x):=f_1(x)-f_2(x)\\
        \nonumber \text{s.t.}\ & x\in \mathbb{R}^n,
    \end{align}
where $f_1, f_2$ are extended convex functions on $\mathbb{R}^n$ and $f$ is an extended  lower-semicontinuous function on $\mathbb{R}^n$.
 Throughout the paper, we assume that the infimum in problem \eqref{P} is finite, and denote by $f^\star$ a lower bound of $f$ on $\mathbb{R}^n$.

DC problems appear naturally in many applications, e.g. power allocation in digital communication systems \cite{alvarado2014new}, production-transportation planning \cite{holmberg1999production},
  location planning \cite{chen1998solution}, image processing \cite{lou2015weighted},
   sparse signal recovering \cite{gasso2009recovering}, cluster analysis \cite{bagirov2018nonsmooth,bagirov2016nonsmooth},
   and supervised data classification \cite{astorino2012margin,le2017dca}, to name but a few.

This wide range of applications is to be expected, since some important classes of nonconvex functions may be represented as DC functions.
  For instance, twice continuously differentiable functions on any convex subset of $\mathbb{R}^n$ \cite{hartman1959functions},
  and continuous piece-wise linear functions \cite{melzer1986expressibility} may be written as DC functions.
  Furthermore, every continuous function on a compact and convex set can be approximated by a DC function \cite{horst1999dc,tuy1998convex}.
  We refer the interested reader to \cite{hiriart1985generalized,tuy1998convex} for more information on DC representable functions.

The celebrated Difference-of-Convex Algorithm (DCA), also known as the convex-concave procedure, has been applied extensively
 to problem \eqref{P}; see \cite{le2018dc,Boyde,tao1997convex} and the references therein.
  Algorithm \ref{Alg} presents the basic form of the DCA.

\begin{algorithm}
\caption{DCA}
\label{Alg}

Pick $x^1\in\mathbb{R}^n$.\\
For $k=1, 2, \ldots$ perform the following steps:
\begin{enumerate}
\item
Choose $g_2^k\in\partial f_2(x^k)$.
\item
Choose
\begin{align}\label{S_P}
  x^{k+1}\in \text{argmin}_{x\in\mathbb{R}^n} f_1(x)-f_2(x^k)-\langle g_2^k, x-x^k\rangle.
\end{align}
\item
If the termination criteria are satisfied, then stop.
\end{enumerate}
\end{algorithm}

%
In the description of the DCA in Algorithm \ref{Alg}, (sub)gradients of $f_1$ and $f_2$ are assumed to be available at given points, the so-called black-box formulation.
The DCA is sometimes also presented as a primal-dual method, where a dual sub-problem is solved to obtain the required (sub)gradients;
 see \cite{le2018dc,Boyde} for further discussions of this topic.
In recent years, some scholars have also extended the DCA  and proposed some new variations; see \cite{gotoh2018dc,lu2019nonmonotone,lu2019enhanced,pang2017computing,sun2021algorithms}.

 The first convergence results for Algorithm \ref{Alg} were given in \cite[Theorem 3(iv)]{tao1997convex}.
 The authors showed that, if the sequence of iterates $\{x^k\}$ is bounded, then each accumulation point of this sequence  is a critical
 point of $f$.

 Le Thi et al. \cite{le2018convergence} established an asymptotic
 linear convergence rate of $\{x^k\}$ under some conditions, in particular under the assumption that
$f$  satisfies the \L{}ojasiewicz  gradient inequality at all stationary points. Recall that a differentiable
 function $f$ is said to satisfy this inequality at a stationary point $a$ ($\nabla f(a) = 0$),
 if there exist constants $\theta \in (0,1)$, $C > 0$ and $\epsilon >0$ such that
\begin{equation}\label{Lojasiewicz}
|f(x) - f(a)|^\theta \le C\|\nabla f(x)\| \mbox{ if $\|x-a\| \le \epsilon$},
\end{equation}
where the constant $\theta$ is called the \L{}ojasiewicz exponent. This inequality is known to hold, for example, for real analytic functions,
 but has been extended to include classes of non-smooth functions as well by considering general sub-differentials instead of gradients; see \cite{Bolte2007,bolte2014proximal},
 and the references therein.

The convergence rates established by Le Thi et al. \cite{le2018convergence} depend on the value of the \L{}ojasiewicz exponent, as the following theorem shows.
The theorem stated here is a special case of Theorems 3.4 and 3.5 in \cite{le2018convergence}, to give a flavor of the convergence results in \cite{le2018convergence}.

\begin{theorem}[Theorems 3.4 and 3.5 in Le Thi et al.\ \cite{le2018convergence}]\label{le2018convergence}
Let $f_1$ and $f_2$ be proper convex functions and let the domain of $f$ be closed.
Also assume that at least one of $f_1$ and $f_2$ is strongly convex, and $f_1$ or $f_2$ is differentiable with locally Lipschitz gradient in every critical point of the DC problem. Finally, assume the sequence $\{x^k\}$ is bounded, and let $x^\infty$ be a limit point of $\{x^k\}$.
Then $x^\infty$ is also a stationary point. Moreover, if $f$ satisfies the \L{}ojasiewicz  gradient inequality \eqref{Lojasiewicz} at all stationary points, then
  \begin{enumerate}
    \item if $\theta\in(1/2,1)$, then $\|x^k-x^\infty\| \leq ck^{\tfrac{1-\theta}{1-2\theta}}$ for some $c>0$.
    \item if $\theta\in(0,1/2]$, then $\|x^k-x^\infty\|\leq cq^k$ for some $c>0$ and $q\in(0,1)$.
  \end{enumerate}
\end{theorem}
In particular, item 2 shows a linear convergence rate when $\theta\in(0,1/2]$.
Yen et al.\ \cite{Yen2012} had already shown linear convergence earlier for a much smaller class of DC functions.
We will present a complementary result to this theorem (see Theorem \ref{theorem_PL} below), for the case $\theta = 1/2$, where we
show linear convergence of the objective function values, and  give explicit expressions for the constants that determine the linear convergence rate.
Moreover, we will relax  the assumption of a bounded sequence of iterates, and the assumption of strong convexity.

In the absence of conditions like the \L{}ojasiewicz  gradient inequality \eqref{Lojasiewicz}, only weaker convergence rates are known for the DCA.
In particular, Tao and An \cite[Proposition 2]{tao1997convex} and Le Thi et al.\ \cite[Corollary 1]{LETHI2021162} have shown an $O\left(\frac{1}{\sqrt{N}} \right)$ convergence rate after $N$ iterations under suitable assumptions,
 as given in the next theorem.
\begin{theorem}[Corollary 1 in \cite{LETHI2021162}, Proposition 2 in \cite{tao1997convex}]\label{thm:le2021convergence}
If $x^\infty$ is a limit point of the iteration sequence generated by the DCA, and at least one of $f_1$ and $f_2$ is strongly convex, i.e.\,
for some $\mu_1, \mu_2 \ge 0$ such that $\mu_1 + \mu_2 > 0$,
\[
x \mapsto f_i(x) - \frac{\mu_i}{2}\|x\|^2 \mbox{ is convex for } i \in \{1,2\},
\]
then the series $\|x^{k+1}-x^k\|$ converges, and, after $N
+1$ iterations,
\[
\sum_{k=1}^{N} \|x^{k+1}-x^k\|^2\leq \tfrac{2(f(x^{1})-f(x^{N+1}))}{\mu_1+\mu_2},
\]
and, consequently,
\[
\min_{1 \le k \le N} \|x^{k+1}-x^k\| \le \sqrt{\frac{2(f(x^1) - f^\star)}{(\mu_1+\mu_2)N} } = O\left(\frac{1}{\sqrt{N}} \right).
\]
\end{theorem}
We will derive some variants on this $O\left(\frac{1}{\sqrt{N}} \right)$ convergence result in Corollary \ref{remark 1} and
 in Section \ref{sec:iterates convergence},
 where we improve the constants in the $O\left(\frac{1}{\sqrt{N}} \right)$ bounds. We also show that we obtain the best possible constants,
 by demonstrating an example where our bound in Corollary \ref{remark 1} is tight.

\subsection*{Outline and further contributions of this paper}


  The novel aspect of the analysis in this paper is that we will apply performance estimation to derive convergence rates.
   Drori and Teboulle, in the seminal paper \cite{drori2014performance},
   introduced performance estimation as a strong tool  for the worst-case analysis of first-order methods. The underlying idea of performance estimation
    is that the worst-case complexity may be cast as an optimization problem. Furthermore, this optimization problem can often be
     reformulated as a semidefinite programming problem. It is worth noting that performance estimation has been employed extensively for
      the analysis of worst-case convergence rates of first-order methods; see, e.g.\
      \cite{abbaszadehpeivasti2021exact,de2020worst,de2017worst,drori2014performance,taylor2017smooth,Taylor}, and the references therein.


\par
This paper is organized  as follows. In Section \ref{Sec.Def} we review some definitions and notions from convex  analysis,
 which will be used in the following sections.
 We study the DCA for sufficiently smooth DC decompositions in Section \ref{Sec.UDCA}. By using performance estimation, we give a convergence
  rate of $O(1/\sqrt{N})$ in Corollary \ref{remark 1}, without any strong convexity assumption,
   thus extending and complementing Le Thi et al.\ \cite[Corollary 1]{LETHI2021162}. We construct an example that shows this $O(1/\sqrt{N})$
   bound is tight.
   Since the first termination criterion is not suitable for the analysis of nonsmooth DC compositions, we investigate the
   DCA with another stopping criterion in Section \ref{sec:nonsmooth criterion},
   and we show a convergence rate of $O(1/N)$. This result is completely new to the best of our knowledge.
 In Section \ref{L_C_PL} we study the DCA when the objective function satisfies the Polyak-{\L}ojasiewicz inequality,
 and we derive a linear convergence rate in Theorem \ref{theorem_PL}, thereby refining some linear convergence results in
 Le Thi et al.\ \cite{le2018convergence} as described above.

\section{Basic Definitions and Preliminaries}\label{Sec.Def}
In this section, we recall some notions and definitions from  convex analysis.
 Throughout the paper, $\|\cdot\|$  and $\langle\cdot,\cdot\rangle$ denote the Euclidean norm and the dot product, respectively.
   $I_{\mathbb{R}_+}$ stands for the indicator function on $\mathbb{R}_+\cup\{\infty\}$, i.e.,
\[
I_{\mathbb{R}_+}(x)=\begin{cases}
1 \ \ \ x\geq 0 \cup\{\infty\}\\
0 \ \ \ x<0\cup\{-\infty\}.
\end{cases}
\]

 Let $f:\mathbb{R}^n\to[-\infty, \infty]$ be an extended convex function.
 The domain of $f$ is denoted and defined as $\text{dom}(f):=\{x: f(x)<\infty\}$.
 The function $f$ is called proper if it does not attain the
value $-\infty$ and its domain is non-empty. We call $f$ closed if its epi-graph is closed,
 that is $\{(x, r): f(x)\leq r\}$ is a closed subset of $\mathbb{R}^{n+1}$.
  We denote the convex hull of $X\subseteq\mathbb{R}^n$ by $\co(X)$.
   We adopt the conventions that, for $a, b, c, d\in\mathbb{R}$ with $c\neq d$ and
   $a\neq 0$, $\frac{b}{\infty}=0, 0\times \infty=0$ and $\frac{a\infty+b}{c\infty-d\infty}=\frac{a}{c-d}$.
    For the function $f:\mathbb{R}^n\to [-\infty, \infty]$,
 the conjugate function $f^*:\mathbb{R}^n\to\mathbb{R}$ is defined as $f^*(g)=\max_{x\in\mathbb{R}^n} \langle g, x\rangle-f(x)$.
  Moreover, we denote the
 set of subgradients of $f$ at $x\in \text{dom}(f)$ by $\partial f(x)$,
$$
\partial f(x)=\{g: f(y)\geq f(x)+\langle g, y-x\rangle, \forall y\in\mathbb{R}^n\}.
$$

Let $L\in(0, \infty]$ and $\mu\in (0, \infty)$. We call an extended convex function $f:\mathbb{R}^n\to[-\infty, \infty]$ $L$-smooth if for any $x_1, x_2\in\mathbb{R}^n$,
$$
\|g_1-g_2\|\leq L\|x_1-x_2\| \ \ \forall  g_1\in\partial f(x_1),\  g_2\in\partial f(x_2).
$$
Note that if $L<\infty$, then $f$ must be differentiable on $\mathbb{R}^n$. In addition, any extended convex function is $\infty$-smooth.
 Also recall that the function $f:\mathbb{R}^n\to[-\infty, \infty]$ is called $\mu$-strongly convex function if
 the function $x \mapsto f(x)-\tfrac{\mu}{2}\| x\|^2$ is convex.
 Clearly, any convex function is $0$-strongly convex.
  We denote the set of closed proper convex functions which are $L$-smooth and $\mu$-strongly convex by $\mathcal{F}_{\mu,L}(\mathbb{R}^n)$.

%

Let $\mathcal{I}$ be a finite index set and let $\{x^i; g^i; f^i\}_{i\in \mathcal{I}}\subseteq \mathbb{R}^n\times \mathbb{R}^n\times \mathbb{R}$.
A set $\{x^i; g^i; f^i\}_{i\in \mathcal{I}}$ is called $\mathcal{F}_{\mu,L}$-interpolable if there exists $f\in\mathcal{F}_{\mu,L}(\mathbb{R}^n)$
 with
$$
f(x^i)=f^i, \ g^i\in\partial f(x^i) \ \ i\in\mathcal{I}.
$$
The next theorem gives necessary and sufficient conditions for $\mathcal{F}_{\mu,L}$-interpolablity.
\begin{theorem}\cite[Theorem 4]{taylor2017smooth}\label{T1}
Let $L\in (0, \infty]$ and $\mu\in [0, \infty)$ and let $\mathcal{I}$ be a finite index set. The set $\{(x^i; g^i; f^i)\}_{i\in \mathcal{I}}\subseteq \mathbb{R}^n\times \mathbb{R}^n \times \mathbb{R}$ is $\mathcal{F}_{\mu,L}$-interpolable if and only if for any $i, j\in\mathcal{I}$, we have
{\small{
\begin{align*}
\tfrac{1}{2(1-\tfrac{\mu}{L})}\left(\tfrac{1}{L}\left\|g^i-g^j\right\|^2+\mu\left\|x^i-x^j\right\|^2-\tfrac{2\mu}{L}\left\langle g^j-g^i,x^j-x^i\right\rangle\right)\leq f^i-f^j-\left\langle g^j, x^i-x^j\right\rangle.
\end{align*}
}}
\end{theorem}

In the next lemma, we extend the descent lemma for DCA when $L_1$ or $L_2$ is finite.
\begin{lemma}\label{Lem_de}
 Let $f_1\in\mathcal{F}_{\mu_1, L_1}({\mathbb{R}^n})$ and $f_2\in\mathcal{F}_{\mu_2, L_2}({\mathbb{R}^n})$ and let $f=f_1-f_2$. If $g_1\in\partial f_1(x)$ and $g_2\in\partial f_2(x)$, then
 $$
 f^\star\leq f(x)-\tfrac{1}{2\left(L_1-\mu_2\right)}\|g_1-g_2\|^2.
 $$
\end{lemma}
{\it Proof }
If $L_1=\infty$, the proof is immediate. Let $L_1<\infty$. By $L$-smoothness and strong convexity, we have
\begin{align*}
f_1(y)\leq f_1(x)+\langle g_1, y-x\rangle+\tfrac{L_1}{2}\|y-x\|^2,\\
f_2(y)\geq f_2(x)+\langle g_2, y-x\rangle+\tfrac{\mu_2}{2}\|y-x\|^2,
\end{align*}
for $y\in\mathbb{R}^n$. By the above inequalities, we get
\begin{align*}
f(y)\leq f(x)+\langle g_1-g_2, y-x\rangle+\tfrac{L_1-\mu_2}{2}\|y-x\|^2.
\end{align*}
Hence, by taking minimum on both sides of the last inequality with respect to $y$ for fixed $x$, we get
$$
f^\star\leq f(x)-\tfrac{1}{2(L_1-\mu_2)}\|g_1-g_2\|^2.
$$
\qed

Since the DC optimization problem \eqref{P} may have a non-convex and non-smooth objective function $f$,
 we will also need a more general notion of subgradients
than in the convex case.
\begin{definition}
 Let $f:\mathbb{R}^n\to\mathbb{R}$ be  lower semi-continuous and let $f(\bar x)$ be finite.
\begin{itemize}
    \item
    The vector $g$ is called regular subgradient of $f$ at $\bar x$, written $g\in \hat\partial_l f(\bar x)$, if for all $x$ in some neighborhood of $\bar x$
    $$
    f(x)\geq f(\bar x)+\langle g, x-\bar x\rangle+o(\|x-\bar x\|).
    $$
      \item
    The vector $g$ is called general subgradient of $f$ at $\bar x$, written
    $g\in \partial_l f(\bar x)$, if there exist sequences $\{x^i\}$ and $\{g^i\}$ with $g^i\in \hat\partial_l f(x^i)$ such that
    $$
    x^i\to \bar x, \ f(x^i)\to f(\bar x), \ g^i\to g.
    $$
\end{itemize}
\end{definition}
It is worth mentioning that $\hat\partial_l f(\bar x)$ is a closed convex set. In addition, $\partial_l f(\bar x)$ is also closed but not necessarily convex. Note that when $f$ is closed proper convex, then $\partial f(x)=\hat\partial_l f(x)=\partial_l f(x)$ for $x\in\text{dom} (f)$. We refer the interested reader to \cite{rockafellar2009variational}
 for more discussions on regular and general subdifferentials.
\begin{definition} \label{Def2}
Let  $f_1, f_2$ be closed proper convex  functions, and let $f$ be lower semi-continuous.
\begin{itemize}
    \item
The point $\bar x\in \text{dom}(f)$  is called a critical point of problem \eqref{P} if
  \begin{align}\label{S.P}
\partial f_1(\bar x)\cap \partial f_2(\bar x) \neq \emptyset.
  \end{align}
  \item
  The point $\bar x\in \text{dom}(f)$  is called a stationary point of problem \eqref{P} if
   \begin{align}\label{S.P2}
  0\in \partial_l f(\bar x).
  \end{align}
  \end{itemize}
  \end{definition}
Obviously, the stationarity condition is stronger than criticality.  We recall that a convex function will be locally Lipschitz
 around $\bar x$ providing it takes finite values in a neighborhood of $\bar x$; see Theorem 35.1 in \cite{Roc}.
  Consequently,  if $f_1$ or $f_2$ takes finite values around a neighborhood of  a stationary point $\bar x$,
  then $\bar x$ is a  critical point; see Corollary 10.9 in \cite{rockafellar2009variational}. However, its converse does not hold in general.
   For instance, consider $f:\mathbb{R}\to\mathbb{R}$ given as $f(x)=x$.
 The function $f$ may be written as $f=f_1-f_2$ where $f_1(x)=\max(x, 0)$ and $f_2(x)=\max(-x, 0)$. Suppose that $\bar x=0$.
 It is readily seen that $\partial f_1(\bar x)\cap \partial f_2(\bar x) \neq \emptyset$, but $\bar x = 0$ is not a stationary point of $f$.
 It is worth noting that, if  $f_2$ is strictly differentiable at $\bar x$,
 these definitions are equivalent; see Example 10.10 in \cite{rockafellar2009variational}.
  Recall that function $f$ is strictly differentiable at $\bar x$, if
  \[
  \lim_{(x, x^\prime)\to (\bar x, \bar x) \atop x\neq x^\prime} \frac{f(x)-f(x^\prime)-\langle \nabla f(\bar x), x-x^\prime\rangle}{\|x-x^\prime\|}=0.
  \]
  We refer the interested reader to \cite{tao2005dc,joki2018double,pang2017computing} and references therein for more discussions on optimality conditions for DC problems.


\subsection{The DC problem}
In this section, we consider
    \begin{align}    \label{P1}
        \min \  &f(x)=f_1(x)-f_2(x)\\
        \nonumber \text{s.t.}\ &x\in\mathbb{R}^n,
    \end{align}
    where $f_1\in\mathcal{F}_{\mu_1, L_1}({\mathbb{R}^n})$ and $f_2\in\mathcal{F}_{\mu_2, L_2}({\mathbb{R}^n})$. Here, we assume that $L_1, L_2\in (0, \infty]$ and $\mu_1, \mu_2\in [0, \infty)$, and consequently $f$ may be non-differentiable.
     We may assume without loss of generality that  $f_1$ and $f_2$ satisfy the following assumptions:
    \begin{align}
    L_1>\mu_2, \ \ \ \ L_2>\mu_1.
    \end{align}
    Indeed, if $L_1\leq \mu_2$, then for $x,y\in\mathbb{R}^n$ and $\lambda\in[0, 1]$, we have
    \begin{align*}
    & \lambda f_1(x)+(1-\lambda)f_1(y)\leq f_1(\lambda x+(1-\lambda)y)+\lambda(1-\lambda)\tfrac{L_1}{2}\|x-y\|^2\\
     &-\lambda f_2(x)-(1-\lambda)f_2(y)\leq -f_2(\lambda x+(1-\lambda)y)-\lambda(1-\lambda)\tfrac{\mu_2}{2}\|x-y\|^2;
    \end{align*}
    see Theorem 2.15 and Theorem 2.19 in \cite{nesterov2018lectures}. By summing the above inequalities, we obtain
    \begin{align*}
    & \lambda f(x)+(1-\lambda)f(y)\leq f(\lambda x+(1-\lambda)y)+\lambda(1-\lambda)\tfrac{L_1-\mu_2}{2}\|x-y\|^2,
    \end{align*}
    which implies concavity of $f$ on $\mathbb{R}^n$. In this case, problem \eqref{P1} will be unbounded from below. This follows from the fact that  a concave function on $\mathbb{R}^n$ is unbounded from below unless it is constant. Likewise, one can show that problem \eqref{P1} will be convex providing $L_2\leq\mu_1$.

The Toland dual \cite{toland1979duality} of problem \eqref{P1} may  be written as
\begin{align}    \label{D1}
        \min \  &f_2^*(x)-f_1^*(x)\\
        \nonumber \text{s.t.}\ &x\in\mathbb{R}^n.
    \end{align}
 It is known that problems \eqref{P1} and \eqref{D1} share the same optimal value \cite{toland1979duality}.

 In what follows, we investigate the convergence rate of Algorithm \ref{Alg}  with the termination criterion $\|g_1^k-g_2^k\|\leq \epsilon$.
  As a motivation of this criterion, recall that $\|g_1^k-g_2^k\| = 0$ implies that $x^k$ is a critical point of \eqref{P} in the non-smooth case,
   and a stationary point of $f$  if  $f_2$ is differentiable; see our discussion following Definition \ref{Def2}.
  In Section \ref{Sec.UDCA} we will derive results for the case that at least one of $f_1$ or $f_2$ is differentiable,
  and we will consider the more general situation in Section \ref{sec:nonsmooth criterion}.

  For well-definedness of the DCA (Algorithm \ref{Alg}), throughout the paper, we assume that
 $$
 x^k\in \text{dom}(\partial f_1)\cap  \text{dom}(\partial f_2) \ \ k=1, 2, \ldots,
 $$
 where $\text{dom}(\partial f_1)=\{x: \partial f_1(x)\neq \emptyset\}$.
 It is worth noting that similar algorithm has been developed for the dual problem in \cite{le2018dc}, and \eqref{S_P} is equivalent to $x^{k+1}\in\partial f^*_1(g_2^k)$.

\section{Performance analysis of the  DCA for smooth $f_1$ or $f_2$}
\label{Sec.UDCA}
In this subsection, we apply performance estimation for the analysis of Algorithm \ref{Alg} for the case that at least one of $f_1$ or $f_2$ is
$L$-smooth for some finite $L>0$.
  The worst-case convergence rate of Algorithm \ref{Alg} can be obtained by solving the following abstract optimization problem:
\begin{align}\label{...P1}
\nonumber   \max & \ \left(\min_{1\leq k\leq N+1} \left\|g_1^k-g_2^k\right\|^2\right)\\
& \nonumber \  g_1^{N+1}, g_2^{N+1},  x^{N+1}, \ldots, x^2 \ \textrm{are generated  by Algorithm \ref{Alg} w.r.t.}\ f_1, f_2, x^1  \\
& f(x)\geq f^\star  \ \ \ \forall x\in\mathbb{R}^n\\
\nonumber  & \ f_1\in \mathcal{F}_{\mu_1,L_1}(\mathbb{R}^n), f_2\in \mathcal{F}_{\mu_2,L_2}(\mathbb{R}^n)\\
\nonumber  &  f_1(x^1)-f_2(x^1)-f^\star\leq \Delta\\
\nonumber  & \ x^1\in\mathbb{R}^n,
\end{align}
where $\Delta\geq0$ denote the difference between the optimal value and the value of $f$ at the starting point.
Here, $f_1, f_2$ and $x^k$, $g_1^k$ and $g_2^k$ ($k\in\{1, ..., N+1\}$) are decision variables, and $\Delta,\mu_1,L_1,\mu_2,L_2$ and $N$ are fixed parameters.

Problem \eqref{...P1} is an intractable infinite-dimensional optimization problem with an
infinite number of constraints. In what follows, we provide a semidefinite programming relaxation of the problem.

By Theorem \ref{T1}, problem \eqref{...P1} can be written as,
\begin{align}\label{P2}
\nonumber   \max & \ \left(\min_{1\leq k\leq N+1} \left\|g_1^k-g_2^k\right\|^2\right)\\
\nonumber \st \ & \tfrac{1}{2(1-\tfrac{\mu_1}{L_1})}\left(\tfrac{1}{L_1}\left\|g_1^i-g_1^j\right\|^2+\mu_1\left\|x^i-x^j\right\|^2-\tfrac{2\mu_1}{L_1}\left\langle g_1^j-g_1^i,x^j-x^i\right\rangle\right)\leq\\
 \nonumber & \ \ \ \ \ f_1^i-f_1^j-\left\langle g_1^j, x^i-x^j\right\rangle \ \ i, j\in\left\{1, \ldots, N+1\right\}  \\
& \nonumber\tfrac{1}{2(1-\tfrac{\mu_2}{L_2})}\left(\tfrac{1}{L_2}\left\|g_2^i-g_2^j\right\|^2+\mu_2\left\|x^i-x^j\right\|^2-\tfrac{2\mu_2}{L_2}\left\langle g_2^j-g_2^i,x^j-x^i\right\rangle\right)\leq\\
& \ \ \ \ \  f_2^i-f_2^j-\left\langle g_2^j, x^i-x^j\right\rangle \ \ i, j\in\left\{1, \ldots, N+1\right\}    \\
\nonumber &  \  g_1^{k+1}=g_2^k  \ \ k\in\left\{1, \ldots, N\right\} \\
\nonumber& \ f_1^k-f_2^k-\frac{1}{2(L_1-\mu_2)}\|g_1^k-g_2^{k}\|^2\geq f^\star  \ \ k\in\left\{1, \ldots, N+1\right\}\\
\nonumber  &  f_1^1-f_2^1-f^\star \leq \Delta.
\end{align}
In problem \eqref{P2}, $f^\star$ and $x^k,\ g_1^k, \ g_2^k, \ f_1^k, \ f_2^k$, $k\in \left\{1, \ldots, N+1\right\}$, are decision variables. By virtue of Lemma \ref{Lem_de}, constraints $f(x)\geq f^\star$ for each $x\in\mathbb{R}^n$ is replaced by $ f_1^k-f_2^k-\frac{1}{2(L_1-\mu_2)}\|g_1^k-g_2^{k}\|^2\geq f^\star,  \ \ k\in\left\{1, \ldots, N+1\right\}$. Due to the necessary and sufficient optimality conditions for convex problems, $x^{k+1}\in \text{argmin}_{x\in\mathbb{R}^n} f_1(x)-f_2(x^k)-\langle g_2^k, x-x^k\rangle$, $k\in\left\{1, \ldots, N\right\}$ implies $g_1^{k+1}=g_2^k$ for some $g_1^{k+1}\in\partial{f}(x^{k+1})$; see Theorem 3.63 in \cite{Amir}.  By substituting  $g_2^{k}=g_1^{k+1}$, $k\in\{1,\ldots, N\}$,
 the above formulation may be written as:
{\small{
\begin{align}\label{P3}
\nonumber   \max & \ \ell\\
\nonumber \st \ & \left\|g_1^i-g_1^{i+1}\right\|^2\geq\ell\ \ \ i\in \{1,\dots, N\}\\
&\nonumber \left\|g_1^{N+1}-g_2^{N+1}\right\|^2\geq\ell\\
&\nonumber \tfrac{1}{2(1-\tfrac{\mu_1}{L_1})}\left(\tfrac{1}{L_1}\left\|g_1^i-g_1^j\right\|^2+\mu_1\left\|x^i-x^j\right\|^2-\tfrac{2\mu_1}{L_1}\left\langle g_1^j-g_1^i,x^j-x^i\right\rangle\right)\leq\\
\nonumber & \ \ \ \ \ f_1^i-f_1^j-\left\langle g_1^j, x^i-x^j\right\rangle \ \ i, j\in\left\{1, \ldots, N+1\right\}  \\
& \nonumber \tfrac{1}{2(1-\tfrac{\mu_2}{L_2})}\left(\tfrac{1}{L_2}\left\|g_1^{i+1}-g_1^{j+1}\right\|^2+\mu_2\left\|x^i-x^j\right\|^2-\tfrac{2\mu_2}{L_2}\left\langle g_1^{j+1}-g_1^{i+1},x^j-x^i\right\rangle\right)\leq\\
 &   \ \ \ \ \ f_2^i-f_2^j-\left\langle g_1^{j+1}, x^i-x^j\right\rangle \ \ i, j\in\left\{1, \ldots, N\right\}    \\  
\nonumber&\tfrac{1}{2(1-\tfrac{\mu_2}{L_2})}\left(\tfrac{1}{L_2}\left\|g_2^{N+1}-g_1^{j+1}\right\|^2+\mu_2\left\|x^{N+1}-x^j\right\|^2-\tfrac{2\mu_2}{L_2}\left\langle g_1^{j+1}-g_2^{N+1},x^j-x^{N+1}\right\rangle\right)\\
\nonumber & \ \ \ \ \ \leq f_2^{N+1}-f_2^j-\left\langle g_1^{j+1}, x^{N+1}-x^j\right\rangle \ \  j\in\left\{1, \ldots, N\right\}\\
& \nonumber\tfrac{1}{2(1-\tfrac{\mu_2}{L_2})}\left(\tfrac{1}{L_2}\left\|g_1^{i+1}-g_2^{N+1}\right\|^2+\mu_2\left\|x^i-x^{N+1}\right\|^2-\tfrac{2\mu_2}{L_2}\left\langle g_2^{N+1}-g_1^{i+1},x^{N+1}-x^i\right\rangle\right)\\
\nonumber & \ \ \ \ \ \leq f_2^i-f_2^{N+1}-\left\langle g_2^{N+1}, x^i-x^{N+1}\right\rangle \ \ i\in\left\{1, \ldots, N\right\}   \\
\nonumber& \ f_1^k-f_2^k-\frac{1}{2(L_1-\mu_2)}\|g_1^k-g_1^{k+1}\|^2\geq f^\star  \ \ k\in\left\{1, \ldots, N\right\}\\
\nonumber& \ f_1^{N+1}-f_2^{N+1}-\frac{1}{2(L_1-\mu_2)}\|g_1^{N+1}-g_2^{N+1}\|^2\geq f^\star\\
\nonumber  &  f_1^1-f_2^1-f^\star \leq \Delta.
\end{align}
}}
By using this formulation, the next result (Theorem \ref{theorem 1}) provides a convergence rate for  Algorithm \ref{Alg}.
Since the proof is quite technical, a few remarks are in order.
The proof uses the performance estimation technique of Drori and Teboulle \cite{drori2014performance}, that consists of the following steps:
\begin{enumerate}
\item
Observe that problem \eqref{P3} may be rewritten as a semidefinite programming (SDP) problem (for sufficiently large $N$) by replacing all inner products
by the entries of an unknown Gram matrix.
\item
Use weak duality of SDP to bound the optimal value of \eqref{P3} by constructing a dual feasible solution.
\item
The dual feasible solution is constructed empirically, by first doing numerical experiments with fixed values of the parameters $\Delta, N, \mu_1, L_1, \mu_2, L_2$, and
noting the dual multipliers.
\item
Subsequently, the analytical expressions of the dual multipliers are guessed, based on the numerical values, and the guess is verified analytically.
\item
In the proof of Theorem \ref{theorem 1}, the conjectured dual multipliers are simply stated, and then shown to provide the required bound
on the optimal value of \eqref{P3} through the corresponding
aggregation of the constraints of \eqref{P3}.
\end{enumerate}

 \begin{theorem}\label{theorem 1}
Let $f_1\in\mathcal{F}_{\mu_1, L_1}({\mathbb{R}^n})$ and $f_2\in\mathcal{F}_{\mu_2, L_2}({\mathbb{R}^n})$ and let $f(x^1)-f^\star= \Delta$. Suppose that
    $L_1$ or $L_2$ is finite. Then after $N$ iterations of Algorithm \ref{Alg}, one has:
\begin{align}\label{B1}
 &\min_{1\leq k\leq N+1}\left\|g_1^k-g_2^k\right\|\leq\sqrt{\frac{\mathcal{A}\Delta}{\mathcal{B} N+\mathcal{C}}},
\end{align}
where,
\begin{align*}
&\mathcal{A}=2\left(L_1L_2-\mu_1L_2I_{\mathbb{R}_+}(L_1-L_2)-\mu_2L_1I_{\mathbb{R}_+}({L_2}-{L_1})\right),\\
& \mathcal{B}=L_1+L_2+\mu_1\left(\tfrac{L_1}{L_2}-3\right)I_{\mathbb{R}_+}\left({L_1}-{L_2}\right)+
\mu_2\left(\tfrac{L_2}{L_1}-3\right)I_{\mathbb{R}_+}\left({L_2}-{L_1}\right),
\end{align*}
and
\[
\mathcal{C}=\frac{L_1L_2-\mu_1L_2I_{\mathbb{R}_+}\left({L_1}-{L_2}\right)-\mu_2L_1I_{\mathbb{R}_+}\left({L_2}-{L_1}\right)}{L_1-\mu_2}.
\]
\end{theorem}
{\it Proof }{
 We investigate two cases $L_1\geq L_2$ and $L_1<L_2$. Suppose that $U$ denote the square of the right-side of inequality \eqref{B1} and let $B=\tfrac{U}{\Delta}$. To prove this bound, we show that $U$ is an upper bound for problem \eqref{P3}.
  First, we  consider $L_1\geq L_2$. Let
\begin{align*}
 &\bar\lambda=\frac{2\left(L_1L_2-\mu_1(2L_2-L_1)\right)}{N\left(L_1+L_2+\mu_1\left(\tfrac{L_1}{L_2}-3\right)\right)+\tfrac{L_2(L_1-\mu_1)}{L_1-\mu_2}}\\
 &\bar\eta_1=\frac{L_2-\mu_1}{\left(L_1+L_2+\mu_1(\tfrac{L_1}{L_2}-3)\right)N+\tfrac{L_2(L_1-\mu_1)}{L_1-\mu_2}}\\
 &\bar\eta_k=\frac{\tfrac{L_1\mu_1}{L_2}+(L_1+L_2-3\mu_1)}{\left(L_1+L_2+\mu_1(\tfrac{L_1}{L_2}-3)\right)N+\tfrac{L_2(L_1-\mu_1)}{L_1-\mu_2}}, \ \ k\in\{2,\ldots,N\}\\
 &\bar\eta_{N+1}=1-\bar\eta_1-\sum_{k=2}^{N}\bar\eta_k=\frac{\tfrac{L_1\mu_1}{L_2}+L_1-2\mu_1+\tfrac{L_2(L_1-\mu_1)}{L_1-\mu_2}}{\left(L_1+L_2+\mu_1(\tfrac{L_1}{L_2}-3)\right)N+\tfrac{L_2(L_1-\mu_1)}{L_1-\mu_2}}.
\end{align*}
 By direct calculation, one can verify that
 {\scriptsize{
\begin{align*}
&\ell-U+\bar\eta_1\left( \left\| g_1^1-g_1^2\right\|^2-\ell\right)+\sum_{k=2}^{N}\bar\eta_k \left( \left\| g_1^k-g_1^{k+1}\right\|^2-\ell\right)+\bar\eta_{N+1}\left( \left\| g_1^{N+1}-g_2^{N+1}\right\|^2-\ell\right)\\
& +B\left( f^\star-f_1^1+f^1_2+\Delta\right)+B\left(f_1^{N+1}- f_2^{N+1}-\frac{1}{2(L_1-\mu_2)}\|g_1^{N+1}-g_2^{N+1}\|^2-f^\star\right)+\\
&B\sum_{k=1}^{N} \Bigg( f_1^k-f_1^{k+1}-\left\langle g_1^{k+1}, x^k-x^{k+1}\right\rangle-\tfrac{1}{2(1-\tfrac{\mu_1}{L_1})}\Bigg(\tfrac{1}{L_1}\left\|g_1^k-g_1^{k+1}\right\|^2+\mu_1\left\|x^k-x^{k+1}\right\|^2-\\
&\tfrac{2\mu_1}{L_1}\left\langle g_1^{k+1}-g_1^k,x^{k+1}-x^k\right\rangle\Bigg)\Bigg)+\bar\lambda\sum_{k=1}^{N-1} \Bigg( f_2^{k+1}-f_2^{k}-\left\langle g_1^{k+1}, x^{k+1}-x^{k}\right\rangle-\\
&\tfrac{1}{2(1-\tfrac{\mu_2}{L_2})}\left(\tfrac{1}{L_2}\left\|g_1^{k+1}-g_1^{k+2}\right\|^2+\mu_2\left\|x^k-x^{k+1}\right\|^2-\tfrac{2\mu_2}{L_2}\langle g_1^{k+2}-g_1^{k+1},x^{k+1}-x^k\rangle\right)\Bigg)\\
&+(\bar\lambda-B)\sum_{k=1}^{N-1} \Bigg( f_2^k-f_2^{k+1}-\left\langle g_1^{k+2}, x^k-x^{k+1}\right\rangle-\tfrac{1}{2(1-\tfrac{\mu_2}{L_2})}(\tfrac{1}{L_2}\left\|g_1^{k+1}-g_1^{k+2}\right\|^2+\\
&\mu_2\left\|x^k-x^{k+1}\right\|^2-\tfrac{2\mu_2}{L_2}\langle g_1^{k+2}-g_1^{k+1},x^{k+1}-x^k\rangle\Bigg)+(\bar\lambda-B)\Bigg( f_2^N-f_2^{N+1}-\left\langle g_2^{N+1}, x^N-x^{N+1}\right\rangle\\
&-\tfrac{1}{2(1-\tfrac{\mu_2}{L_2})}\left(\tfrac{1}{L_2}\left\|g_1^{N+1}-g_2^{N+1}\right\|^2+\mu_2\left\|x^N-x^{N+1}\right\|^2-\tfrac{2\mu_2}{L_2}\langle g_2^{N+1}-g_1^{N+1},x^{N+1}-x^N\rangle\right)\Bigg)\\
& +\bar\lambda\Bigg( f_2^{N+1}-f_2^{N}-\left\langle g_1^{N+1}, x^{N+1}-x^{N}\right\rangle-\tfrac{1}{2(1-\tfrac{\mu_2}{L_2})}\Bigg(\tfrac{1}{L_2}\left\|g_1^{N+1}-g_2^{N+1}\right\|^2+\mu_2\left\|x^N-x^{N+1}\right\|^2\\
&-\tfrac{2\mu_2}{L_2}\langle g_2^{N+1}-g_1^{N+1},x^{N+1}-x^N\rangle\Bigg)\Bigg)
=\\
&-\bar\beta_{1}^{-1}\sum_{i=1}^{N}\left\|\bar\beta_{1}g^{i}_1-\bar\beta_{1}g_1^{i+1}-\bar\alpha_{1}x^{i}+\bar\alpha_{1}x^{i+1}\right\|^2-\bar\alpha_{2}^{-1}\sum_{i=1}^{N-1}\left\|\bar\alpha_{2}x^{i}-\bar\alpha_{2}x^{i+1}-\bar\beta_{2}g^{i+1}_1+\bar\beta_{2}g^{i+2}_1\right\|^2\\
&-\bar\alpha_{2}^{-1}\left\|\bar\alpha_{2}x^{N}-\bar\alpha_{2}x^{N+1}-\bar\beta_{2}g^{N+1}_1+\bar\beta_{2}g^{N+1}_2\right\|^2
\leq0,
\end{align*}
}}
where
\begin{align*}
  &\bar\alpha_1=\frac{\mu_1 B}{2(L_1-\mu_1)}, \ \  \ \ \bar\beta_1=\frac{\mu_1B}{2L_2(L_1-\mu_1)}, \\
   &   \bar\alpha_2=\frac{(-\mu_1L_2^2-2\mu_1\mu_2L_2+\mu_1L_1L_2+\mu_1\mu_2L_1+\mu_2L_1L_2)B }{2(L_1-\mu_1)(L_2-\mu_2)},\\
    &  \bar\beta_2=\frac{(L_1L_2\mu_2-2\mu_1\mu_2L_2+\mu_1\mu_2L_1-\mu_1L_2^2+\mu_1L_1L_2)B}{2L_2(L_1-\mu_1)(L_2-\mu_2)}.
\end{align*}
It is readily seen that $\bar\lambda, \bar\eta_k\ (k\in\{1, \ldots, N+1\}), \bar\lambda-B, \bar\beta_1, \bar\alpha_2\geq 0$.
Thus we have $\ell\leq U$ for any feasible point of problem \eqref{P3}.
Now, we consider $L_1<L_2$. In this case, because bound \eqref{B1} does not depend on $\mu_1$, we may assume  $\mu_1=0$ in problem \eqref{P3}.
Let
\begin{align*}
 &\hat\lambda=\frac{2\left(L_1L_2-\mu_2(2L_1-L_2)\right)}{\left(L_1+L_2+\mu_2\left(\tfrac{L_2}{L_1}-3\right)\right)N+\tfrac{L_1(L_2-\mu_2)}{L_1-\mu_2}}\\
 &\hat\eta_1=\frac{\tfrac{L_2(L_1+\mu_2)}{L_1}-2\mu_2}{\left(L_1+L_2+\mu_2(\tfrac{L_2}{L_1}-3)\right)N+\tfrac{L_1(L_2-\mu_2)}{L_1-\mu_2}}\\
 &\hat\eta_k=\frac{\tfrac{L_2(L_1+\mu_2)}{L_1}+(L_1-3\mu_2)}{\left(L_1+L_2+\mu_2(\tfrac{L_2}{L_1}-3)\right)N+\tfrac{L_1(L_2-\mu_2)}{L_1-\mu_2}}, \ \ k\in\{2,\ldots,N\}\\
 &\hat\eta_{N+1}=1-\hat\eta_1-\sum_{k=2}^{N}\hat\eta_k=\frac{\tfrac{L_1(L_2-\mu_2)}{L_1-\mu_2}+L_1-\mu_2}{\left(L_1+L_2+\mu_2(\tfrac{L_2}{L_1}-3)\right)N+\tfrac{L_1(L_2-\mu_2)}{L_1-\mu_2}}.
\end{align*}
 With some calculation, one can establish that
 {\scriptsize{
\begin{align*}
&\ell-U+\hat\eta_1\left( \left\| g_1^1-g_1^2\right\|^2-\ell\right)+\sum_{k=2}^{N} \hat\eta_k\left( \left\| g_1^k-g_1^{k+1}\right\|^2-\ell\right)+\hat\eta_{N+1}\left( \left\| g_1^{N+1}-g_2^{N+1}\right\|^2-\ell\right) \\
&+ B\left( f^\star-f_1^1+f^1_2+\Delta\right)+B\left(f_1^{N+1}- f_2^{N+1}-\frac{1}{2(L_1-\mu_2)}\|g_1^{N+1}-g_2^{N+1}\|^2-f^\star\right)\\
& +(\hat\lambda-B)\sum_{k=1}^{N} \left( f_1^{k+1}-f_1^{k}-\left\langle g_1^{k}, x^{k+1}-x^{k}\right\rangle-\tfrac{1}{2L_1}\left\|g_1^{k+1}-g_1^{k}\right\|^2\right)\\
&+\hat\lambda\sum_{k=1}^{N} \left( f_1^k-f_1^{k+1}-\left\langle g_1^{k+1}, x^k-x^{k+1}\right\rangle-\tfrac{1}{2L_1}\left\|g_1^{k}-g_1^{k+1}\right\|^2\right)\\
&+B\sum_{k=1}^{N-1} \Bigg( f_2^{k+1}-f_2^{k}-\left\langle g_1^{k+1}, x^{k+1}-x^{k}\right\rangle-\\
&\tfrac{1}{2(1-\tfrac{\mu_2}{L_2})}\left(\tfrac{1}{L_2}\left\|g_1^{k+1}-g_1^{k+2}\right\|^2+\mu_2\left\|x^k-x^{k+1}\right\|^2-\tfrac{2\mu_2}{L_2}\left\langle g_1^{k+2}-g_1^{k+1},x^{k+1}-x^k\right\rangle\right)\Bigg)\\
&+B\Bigg( f_2^{N+1}-f_2^{N}-\left\langle g_1^{N+1}, x^{N+1}-x^{N}\right\rangle-\\
&\tfrac{1}{2(1-\tfrac{\mu_2}{L_2})}\left(\tfrac{1}{L_2}\left\|g_1^{N+1}-g_2^{N+1}\right\|^2+\mu_2\left\|x^N-x^{N+1}\right\|^2-\tfrac{2\mu_2}{L_2}\left\langle g_2^{N+1}-g_1^{N+1},x^{N+1}-x^N\right\rangle\right)\Bigg)\\
&= -\hat\beta_{1}^{-1}\sum_{i=1}^{N}\left\|\hat\beta_{1}g^{i}_1-\hat\beta_{1}g_1^{i+1}-\hat\alpha_{1}x^{i}_1+\hat\alpha_{1}x^{i+1}\right\|^2
-\hat\alpha_{2}^{-1}\sum_{i=1}^{N-1}\left\|\hat\alpha_{2}x^{i}-\hat\alpha_{2}x^{i+1}-\hat\beta_{2}g^{i+1}_1+\hat\beta_{2}g^{i+2}_1\right\|^2\\
&-\hat\alpha_{2}^{-1}\left\|\hat\alpha_{2}x^{N}-\hat\alpha_{2}x^{N+1}-\hat\beta_{2}g^{N+1}_1+\hat\beta_{2}g^{N+1}_2\right\|^2
\leq0,
\end{align*}
}}
where
\begin{equation*}
\hat\alpha_1=\tfrac{\mu_2B(1-\tfrac{L_1}{L_2})}{2L_1(1-\tfrac{\mu_2}{L_2})}, \ \
 \hat\alpha_2=\tfrac{\mu_2L_1B}{2(L_2-\mu_2)}, \ \
    \hat\beta_1=\tfrac{\mu_2B(1-\tfrac{L_1}{L_2})}{2L_1^2(1-\tfrac{\mu_2}{L_2})}, \ \
    \hat\beta_2=\tfrac{\mu_2B}{2(L_2-\mu_2)}.
\end{equation*}
It is readily seen that $\hat\lambda, \hat\eta_k\ (k\in\{1, \ldots, N+1\}), \hat\lambda-B, \hat\beta_1, \hat\alpha_2\geq 0$. The rest of proof is similar to that of the former case, and the proof is complete.
\qed
The theorem implies that Algorithm \ref{Alg} is convergent when at least one of the Lipschitz constants is finite. In the  following corollary, we simplify the inequality \eqref{B1} for some special cases of $L_1$, $L_2$, $\mu_1$, and $\mu_2$.

\begin{corollary}\label{remark 1}
Suppose that $f_1\in\mathcal{F}_{\mu_1, L_1}({\mathbb{R}^n})$ and $f_2\in\mathcal{F}_{\mu_2, L_2}({\mathbb{R}^n})$.
Then, after $N$ iterations of Algorithm \ref{Alg}, one has:
\begin{enumerate}[i)]
    \item
    If $L_1=\infty$, $L_2<\infty$, then
    \[
    \min_{1\leq k\leq N+1}\left\|g_1^k-g_2^k\right\|\leq \sqrt{\frac{2L_2^2\left( f(x^1)-f^\star \right)}{N(L_2+\mu_1)}}.
\]
\item
    If $L_2=\infty$, $L_1<\infty$, then
    \begin{equation}
\label{bound Cor 31ii}
    \min_{1\leq k\leq N+1}\left\|g_1^k-g_2^k\right\|\leq \sqrt{\frac{2L_1^2\left(L_1-\mu_2\right)\left( f(x^1)-f^\star \right)}{\left(L_1^2-\mu_2^2\right)N+L_1^2}}.
\end{equation}
\item
    If $L_1, L_2<\infty$, and $\mu_1=\mu_2=0$ then
    \[
    \min_{1\leq k\leq N+1}\left\|g_1^k-g_2^k\right\|\leq \sqrt{\frac{2L_1L_2\left(f(x^1)-f^\star\right)}{\left(L_1+L_2\right)N+L_2}}.
\]
\end{enumerate}
\end{corollary}


One can compare the results in Corollary \ref{remark 1} to that of Le Thi et al. \cite{LETHI2021162} as reviewed
 earlier in Theorem \ref{thm:le2021convergence}.
 First of all, Corollary \ref{remark 1} part $iii)$ does not assume strict convexity of $f_1$ or $f_2$, and in this sense it is more general than the
 result in Theorem \ref{thm:le2021convergence}. If we do assume $\mu_1+\mu_2 > 0$, then, for example, if $L_1 < \infty$,
  Theorem \ref{thm:le2021convergence} implies,
 \[
  \min_{1\leq k\leq N+1}\left\|g_1^k-g_2^k\right\|\leq L_1\sqrt{\frac{2\left( f(x^1)-f^\star \right)}{{\left(\mu_1+\mu_2\right)N}}},
 \]
which is weaker than our bound \eqref{bound Cor 31ii} since $\mu_1 \le L_1$, although the $O(1/\sqrt{N})$ dependence on $N$ is the same.
We will do a further, more direct, comparison of Theorem \ref{thm:le2021convergence} and Corollary \ref{remark 1} in Section \ref{sec:iterates convergence},
where we consider the convergence rate of the sequence $\|x^{k+1} - x^k\|$.

\subsection{An example to prove tightness}
  In what follows, we give a class of functions for which  the bound in Corollary \ref{remark 1}, part $ii)$, is attained,
  implying that the $O(1/\sqrt{N})$ convergence rate is tight.
  This result is new to the best of our knowledge.
\begin{example}
Let $L_1\in(0, \infty)$. Suppose that $N$ is selected such that $U:=\sqrt{\tfrac{2}{L_1(N+1)}}< 1$. Let $f_1: \mathbb{R}\to\mathbb{R}$ be given as follows,
\begin{align*}
f_1(x) = \begin{cases}
\tfrac{L_1}{2}\left(x-i(1-U)\right)^2+\tfrac{L_1Ui(i-1)(1-U)}{2}  &  \ \ x\in\left[\alpha_i, \beta_{i+1}\right)\\
L_1U\beta_i(x-\beta_i)+\tfrac{\beta_iL_1U^2}{2}+\tfrac{\beta_i(\beta_i-1)L_1U}{2} &  \ \ x\in \left[\beta_{i}, \alpha_i\right)\\
\tfrac{L_1}{2}x^2 &  \ \ x\in\left(-\infty, 0\right),
\end{cases}
\end{align*}
where, for $i\in\{1, \ldots, N+1\}$,
    $\alpha_i=i-U$, $\beta_i=i-1$,
and $\beta_{N+2}=\infty$.
Note that $f_1\in\mathcal{F}_{0, L_1}({\mathbb{R}})$. Suppose that $f_2: \mathbb{R}\to\mathbb{R}$ is given by
\begin{align*}
    f_2(x)=\max_{1\leq i\leq N+1}\left\{L_1U(i-1)(x-i)+\tfrac{i(i-1)L_1U}{2}\right\}.
\end{align*}
An easy computation shows that
$$
\begin{cases}
    \partial f_2(i)=[L_1U(i-1), L_1Ui]  &  \ \ \ i\in\{1,\dots,N,\}\\
     \partial f_2(N+1)=L_1UN. \\

    \end{cases}
$$
Note that $f_2\in\mathcal{F}_{0, \infty}({\mathbb{R}})$. One can check that, at $x^1=N+1$, one has $f_1(x^1)-f_2(x^1)=1$, $\min_{x\in\mathbb{R}} f_1(x)-f_2(x)=0$ and $\text{argmin}_{x\in\mathbb{R}} f_1(x)-f_2(x)=[0, 1-U]$. By taking $x^1$ as a starting point, Algorithm \ref{Alg}  can generate the following iterates:
$$
x^k=N+2-k, \ \ \ \ k\in\{1, \ldots, N+1\}.
$$
Here at iteration, $k\in\{1, \ldots, N+1\}$, we set $g_2^k=L_1U(N+1-k)$.
It follows that  $|\nabla f_1(x_k)-g_2^k|=\sqrt{\tfrac{2L_1}{N+1}}$, $k\in\{1,\ldots,N+1\}$.
Hence,
$$
\min_{1\leq k\leq N+1}\left\|g_1^k-g_2^k\right\|=\sqrt{\tfrac{2L_1}{N+1}},
$$
 which shows bound \eqref{bound Cor 31ii} in Corollary \ref{remark 1} is exact for this example.
\end{example}

\subsection{Convergence rates for the iterates}
\label{sec:iterates convergence}
In this section we investigate the implications of our results so far on convergence rates
of the iterates $\{x^k\}$.

 \begin{proposition}\label{Por_N1}
Let $f_1\in\mathcal{F}_{\mu_1, L_1}({\mathbb{R}^n})$ and $f_2\in\mathcal{F}_{\mu_2, L_2}({\mathbb{R}^n})$ and let $f(x^1)-f^\star \leq \Delta$. If
    $\mu_1$ or $\mu_2$ is strictly positive, then after $N$ iterations of Algorithm \ref{Alg}, one has:
\begin{align*}
 &\min_{1\leq k\leq N}\left\|x^{k+1}-x^k\right\|\leq\left(\frac{\mathcal{A}}{\mathcal{B} N+\mathcal{C}}\cdot\Delta\right)^{\tfrac{1}{2}},
\end{align*}
where,
{\small
\begin{align*}
&\mathcal{A}=2\left(\mu_2^{-1}\mu_1^{-1}-L_2^{-1}\mu_1^{-1}I_{\mathbb{R}_+}(\mu_2^{-1}-\mu_1^{-1})-L_1^{-1}\mu_2^{-1}I_{\mathbb{R}_+}({\mu_1^{-1}}-{\mu_2^{-1}})\right),\\
&\mathcal{B}=\mu_2^{-1}+\mu_1^{-1}+L_2^{-1}\left(\tfrac{\mu_1}{\mu_2}-3\right)I_{\mathbb{R}_+}\left({\mu_2^{-1}}-{\mu_1^{-1}}\right)+
L_1^{-1}\left(\tfrac{\mu_2}{\mu_1}-3\right)I_{\mathbb{R}_+}\left({\mu_1^{-1}}-{\mu_2^{-1}}\right),\\
&\text{and}\\
&\mathcal{C}=\frac{\mu_2^{-1}\mu_1^{-1}-L_2^{-1}\mu_1^{-1}I_{\mathbb{R}_+}\left({\mu_2^{-1}}-{\mu_1^{-1}}\right)-L_1^{-1}\mu_2^{-1}I_{\mathbb{R}_+}\left({\mu_1^{-1}}-{\mu_2^{-1}}\right)}{\mu_2^{-1}-L_1^{-1}}.
\end{align*}
}
\end{proposition}
{\it Proof }
 The proof is based on the computation of the worst case convergence rate of DCA for problem  \eqref{D1} by applying Theorem \ref{theorem 1}.
 By  Toland duality, $f^\star $ is also a lower bound of problem \eqref{D1}.
 By virtue of conjugate function properties, it follows that $ f_2^*(g_2^1)-f_1^*(g_2^1)-f^\star\leq\Delta$ and $f_2^*\in\mathcal{F}_{L_2^{-1},
  \mu_2^{-1}}({\mathbb{R}^n})$ and $f_1^*\in\mathcal{F}_{L_1^{-1}, \mu_1^{-1}}({\mathbb{R}^n})$.
   In addition, $x^{k+1}\in\partial f_1^*(g_2^k)$ and $x^{k}\in\partial f_2^*(g_2^k)$ for $k\in\{1, \ldots, N\}$.
   Hence, all assumptions of Theorem \ref{theorem 1} hold, and subsequently the bound follows from Theorem \ref{theorem 1}.
\qed

Recall the known result from Theorem \ref{thm:le2021convergence}:
\begin{align}\label{B_X_s}
 \min_{1\leq k\leq N}\left\|x^{k+1}-x^k\right\|\leq
 \left(\frac{2(f(x^1)-f^\star)}{N(\mu_1+\mu_2)}\right)^{\tfrac{1}{2}}.
\end{align}
By employing Theorem \ref{Por_N1}, we get
\begin{align*}
 \min_{1\leq k\leq N}\left\|x^{k+1}-x^k\right\|\leq
 \left(\frac{2(f(x^1)-f^\star)}{N(\mu_1+\mu_2)+\mu_1}\right)^{\tfrac{1}{2}},
\end{align*}
which is tighter than the bound \eqref{B_X_s}.  Moreover, the bound given in Proposition \ref{Por_N1} provides more information
 concerning the worst-case convergence rate of the DCA when $L_1<\infty$ or $L_2<\infty$.

\section{Performance estimation using a convergence criterion for critical points in the nonsmooth case}
\label{sec:nonsmooth criterion}
 Theorem \ref{theorem 1} addresses  the case that $f_1$ or $f_2$ is $L$-smooth with $L<\infty$. In what follows,
 we investigate the case that $f_1$ and $f_2$ are proper convex functions and where both may be non-smooth. For this general case,
  we need to adopt a different termination criterion to obtain results, since the termination criterion $\|g_1^k-g_2^k\|\leq \epsilon$  may be of no use
  in this case. For example, suppose that a DC function $f: \mathbb{R}\to\mathbb{R}\cup \{\infty\}$ is given by
 $$
 f(x)=\begin{cases}
        f_1(x)-f_2(x) & x\geq 0 \\
        \infty & x<0,
      \end{cases}
 $$
 where
 \begin{align*}
  & f_1(x)=\max_{n\in\mathbb{N}\cup\{0\}}\{-n(x-2^{-n})+2-2^{1-n}-n2^{-n} \},\\
     &    f_2(x)=\max_{n\in\mathbb{N}\cup\{0\}} \{-(n+1)(x-2^{-n})+2-3(2^{-n})-n2^{-n} \}.
 \end{align*}
With $x^1=1$ and the given DC decomposition, Algorithm \ref{Alg} may generate
\begin{align*}
  x^k=2^{-k}, \ \ \ \ g_1^k=-(k-1), \ \ \ \  g_2^k=-k,  \ \ \ k\in\{1, 2, ...\}.
\end{align*}
As $|g^k_1-g_2^k|=1$, Algorithm \ref{Alg} never stops by employing the given termination criterion while it is convergent to global minimum $\bar x=0$.
We therefore will use the termination criterion of the following value being sufficiently small:
 \begin{align}\label{termination2}
\nonumber T(x^{k+1})&:=f_1(x^k)-f_2(x^k)-\min_{x\in \mathbb{R}^n} \left(f_1(x)-f_2(x^k)-\left\langle g_2^k, x-x^k\right\rangle\right)\\&=f_1(x^k)-f_1(x^{k+1})-\left\langle g_2^k, x^k-x^{k+1}\right\rangle.
\end{align}
Note that $T(x^{k+1})\geq 0$.
It follows that if $T(x^{k+1})=0$ then $f(x^k)= f(x^{k+1})$, and $x^{k}\in \text{argmin}_{x\in \mathbb{R}^n} f_1(x)-f_2(x^k)-\langle g_2^k,
 x-x^k\rangle$. Indeed, by the optimality conditions for convex problems, we have $\partial f_1(x^k)\cap \partial f_2(x^k)\neq \emptyset$.
  Consequently, $T(x^{k+1})=0$ implies that $x^{k}$ is a critical point of problem \eqref{P1}.
  The aforementioned stopping criterion has also been employed for the analysis of the  Frank-Wolfe method for nonconvex problems;
   see equation (2.6) in \cite{ghadimi2019conditional}.

   In what follows, we investigate Algorithm \ref{Alg} with the termination criterion $T(x^{k+1})<\epsilon$ for the given accuracy $\epsilon>0$.
    The performance estimation problem with termination criterion \eqref{termination2} may be written as follows,
{\scriptsize{
\begin{align}\label{P3.A2}
\nonumber   \max & \ \ell\\
\nonumber \st \ & f_1(x^k)-f_1(x^{k+1})-\left\langle g_1^{k+1}, x^k-x^{k+1}\right\rangle\geq\ell\ \ \ i\in \{1,\dots, N\}\\
&\nonumber \tfrac{1}{2(1-\tfrac{\mu_1}{L_1})}\left(\tfrac{1}{L_1}\left\|g_1^i-g_1^j\right\|^2+\mu_1\left\|x^i-x^j\right\|^2-\tfrac{2\mu_1}{L_1}\left\langle g_1^j-g_1^i,x^j-x^i\right\rangle\right)\\
\nonumber&\leq f_1^i-f_1^j-\left\langle g_1^j, x^i-x^j\right\rangle \ \ i, j\in\left\{1, \ldots, N+1\right\}  \\
\nonumber&\tfrac{1}{2(1-\tfrac{\mu_2}{L_2})}\left(\tfrac{1}{L_2}\left\|g_1^{i+1}-g_1^{j+1}\right\|^2+\mu_2\left\|x^i-x^j\right\|^2-\tfrac{2\mu_2}{L_2}\left\langle g_1^{j+1}-g_1^{i+1},x^j-x^i\right\rangle\right)\\
&\leq \ f_2^i-f_2^j-\left\langle g_1^{j+1}, x^i-x^j\right\rangle \ \ i, j\in\left\{1, \ldots, N\right\}    \\  
\nonumber&\tfrac{1}{2(1-\tfrac{\mu_2}{L_2})}\left(\tfrac{1}{L_2}\left\|g_2^{N+1}-g_1^{j+1}\right\|^2+\mu_2\left\|x^{N+1}-x^j\right\|^2-\tfrac{2\mu_2}{L_2}\left\langle g_1^{j+1}-g_2^{N+1},x^j-x^{N+1}\right\rangle\right)\\
\nonumber&\leq  f_2^{N+1}-f_2^j-\left\langle g_1^{j+1}, x^{N+1}-x^j\right\rangle \ \  j\in\left\{1, \ldots, N\right\}\\
& \nonumber\tfrac{1}{2(1-\tfrac{\mu_2}{L_2})}\left(\tfrac{1}{L_2}\left\|g_1^{i+1}-g_2^{N+1}\right\|^2+\mu_2\left\|x^i-x^{N+1}\right\|^2-\tfrac{2\mu_2}{L_2}\left\langle g_2^{N+1}-g_1^{i+1},x^{N+1}-x^i\right\rangle\right)\\
\nonumber&\leq f_2^i-f_2^{N+1}-\left\langle g_2^{N+1}, x^i-x^j\right\rangle \ \ i\in\left\{1, \ldots, N\right\}   \\
\nonumber& \ f_1^k-f_2^k\geq f^\star  \ \ k\in\left\{1, \ldots, N+1\right\}\\
\nonumber  &  f_1^1-f_2^1-f^\star \leq \Delta.
\end{align}
}}

Note that we do not employ Lemma \ref{Lem_de} in this formulation because we consider a general DC problem. Using the performance estimation procedure as described before the proof  of  Theorem \ref{theorem 1} once more, we obtain the following result.
\begin{theorem}\label{theorem 2}
Let $f_1\in\mathcal{F}_{\mu_1, L_1}({\mathbb{R}^n})$ and $f_2\in\mathcal{F}_{\mu_2, L_2}({\mathbb{R}^n})$. Then, after $N$ iterations
of Algorithm \ref{Alg}, one has
\begin{align}
&\min_{1\leq k\leq N} f_1(x^k)-f_1(x^{k+1})-\langle g_2^k, x^k-x^{k+1}\rangle\leq \\
&\nonumber \min\left\{\frac{L_1}{N(L_1+\mu_2)}, \frac{L_2}{N(L_2+\mu_1)-\mu_1} \right\}\left(f(x^1)-f^\star \right).
\end{align}
\end{theorem}
{\it Proof }
 We show separately that $\tfrac{L_1(f(x^1)-f^\star )}{N(L_1+\mu_2)}$ and $\frac{L_2(f(x^1)-f^\star )}{N(L_2+\mu_1)-\mu_1}$ are upper bounds for problem \eqref{P3.A2}. The proof is analogous to that of Theorem \ref{theorem 1}. First, consider the bound $\tfrac{L_1(f(x^1)-f^\star )}{N(L_1+\mu_2)}$. Since the given bound does not depend on $\mu_1$ and $L_2$, we may assume without loss of generality that $L_2=\infty$ and $\mu_1=0$. Suppose that $B_1=\tfrac{L_1}{N(L_1+\mu_2)}$.
 With some algebra, one can show that
 {\scriptsize{
\begin{align*}
&\ell-B_1\Delta+\tfrac{1}{N}\sum_{k=1}^{N} \left( f_1^k-f_1^{k+1}-\langle g_1^{k+1}, x^k-x^{k+1}\rangle-\ell\right)+B_1\left(f_1^{N+1}- f_2^{N+1}-f^\star\right)+\\
&  B_1\left( f^\star-f_1^1+f^1_2+\Delta\right)
 +(\tfrac{1}{N}-B_1)\sum_{k=1}^{N} \left( f_1^{k+1}-f_1^{k}-\left\langle g_1^{k}, x^{k+1}-x^{k}\right\rangle-\tfrac{1}{2L_1}\left\|g_1^{k+1}-g_1^{k}\right\|^2\right)\\
& +B_1\sum_{k=1}^{N} \left( f_2^{k+1}-f_2^{k}-\left\langle g_1^{k+1}, x^{k+1}-x^{k}\right\rangle-\tfrac{\mu_2}{2}\left\|x^{k+1}-x^{k}\right\|^2\right)\\
&= -\tfrac{B_1\mu_2}{2}\sum_{k=1}^N\left\|x^k-x^{k+1}-\tfrac{1}{L_1}(g_1^k-g_1^{k+1})\right\|^2
\leq0.
\end{align*}
}}
The rest of proof is similar to that of  Theorem \ref{theorem 1}. Now, we consider the bound $\tfrac{L_2(f(x^1)-f^\star )}{N(L_2+\mu_1)-\mu_1}$. Without loss generality, we may assume that $L_1=\infty$ and $\mu_2=0$. By doing some calculus, one can show that
 {\small{
\begin{align*}
&\ell-B_2\Delta+ B_2\left( f_1^1-f_1^{2}-\left\langle g^2_1, x^1-x^{2}\right\rangle-\ell\right)+ B_2\left(f_1^{N+1}- f_2^{N+1}-f^\star\right) \\
&+B_2\left( f^\star-f_1^1+f^1_2+\Delta\right)+\tfrac{1-B_2}{N-1}\sum_{k=2}^{N} \left( f_1^k-f_1^{k+1}-\left\langle g_1^{k+1}, x^k-x^{k+1}\right\rangle-\ell\right) \\ & +\alpha\sum_{k=2}^{N} \left( f_1^{k+1}-f_1^{k}-\left\langle g_1^{k}, x^{k+1}-x^{k}\right\rangle-\tfrac{\mu_1}{2}\left\|x^{k+1}-x^{k}\right\|^2\right)\\
& +B_2\sum_{k=1}^{N} \left( f_2^{k+1}-f_2^{k}-\left\langle g_1^{k+1}, x^{k+1}-x^{k}\right\rangle-\tfrac{1}{2L_2}\left\|g_1^{k+2}-g_1^{k+1}\right\|^2\right)\\
& +B_2\left( f_2^{N+1}-f_2^{N}-\left\langle g_1^{N+1}, x^{N+1}-x^{N}\right\rangle-\tfrac{1}{2L_2}\left\|g_2^{N+1}-g_1^{N+1}\right\|^2\right)\\
&=-\tfrac{B_2}{2L_2}\left\|g_2^{N+1}-g_1^{N+1}\right\|^2-\tfrac{B_2}{2L_2}\sum_{k=2}^{N} \left\|g_1^k-g_1^{k+1}-\tfrac{\alpha L_2}{B_2}(x^{k}-x^{k+1})  \right\|^2
\leq0,
\end{align*}
}}
where $B_2=\tfrac{L_2}{N(L_2+\mu_1)-\mu_1}$ and $\alpha=\tfrac{1-B_2}{N-1}-B_2$.  Since we assume $L_2>\mu_1$, we have $B_2, \alpha\geq 0$. The rest of the proof runs as before.
\qed

The important point is that the last result provides a rate of convergence even if neither $L_1$ nor $L_2$ is finite, and we therefore state it as
a corollary.

\begin{corollary}
Let $f_1\in\mathcal{F}_{\mu_1, \infty}({\mathbb{R}^n})$ and $f_2\in\mathcal{F}_{\mu_2, \infty}({\mathbb{R}^n})$, i.e.\ consider any
 DC decomposition in problem \eqref{P}. Then, after $N$ iterations
of Algorithm \ref{Alg}, one has
\[
\min_{1\leq k\leq N} f_1(x^k)-f_1(x^{k+1})-\langle g_2^k, x^k-x^{k+1}\rangle\leq
 \frac{1}{N} \left(f(x^1)-f^\star \right).
\]
\end{corollary}
This result is new to the best of our knowledge.

\section{Linear convergence of the DCA under the  Polyak-\L ojasiewicz  inequality}\label{L_C_PL}
In the section, we provide some sufficient conditions under which the DCA is linearly convergent. Similar to the former sections, we employ the performance estimation for obtaining convergence rate.

In recent years, the linear convergence of some optimization methods for non-convex problems have been investigated under
 the  Polyak-\L ojasiewicz (PL) inequality; see \cite{abbaszadehpeivasti2022PL,bolte2017eror,Karimi} and the reference therein.
  We say that $f$ satisfies PL inequality on $X$  if there exists $\eta>0$ such that
\begin{align}\label{PL_i}
 f(x)-f^\star\leq \tfrac{1}{2\eta} \| \xi \|^2, \ \ \forall x\in X, \forall \xi\in\co(\partial_l f(x)).
 \end{align}
Note that when $f$ is differentiable inequality \eqref{PL_i} is a special case of \eqref{Lojasiewicz}
with $\theta=\tfrac{1}{2}$ and different ground set.
If $f_1$ or $f_2$ is strictly differentiable,
we have have $\co(\partial_l f)=\partial f_1-\partial f_2$; see Example 10.10 in \cite{rockafellar2009variational}.
 Hence, the performance estimation problem with the PL inequality may be formulated as follows:
{{
\begin{align}\label{P33}
\nonumber   \max & \ \frac{(f_1^2-f_2^2)-f^\star}{(f_1^1-f_2^1)-f^\star}\\
\nonumber \st \ & \tfrac{1}{2(1-\tfrac{\mu_1}{L_1})}\left(\tfrac{1}{L_1}\left\|g_1^i-g_1^j\right\|^2+\mu_1\left\|x^i-x^j\right\|^2-\tfrac{2\mu_1}{L_1}\left\langle g_1^j-g_1^i,x^j-x^i\right\rangle\right)\\
&\nonumber\leq f_1^i-f_1^j-\left\langle g_1^j, x^i-x^j\right\rangle \ \ i, j\in\left\{1, 2\right\}  \\
& \nonumber\tfrac{1}{2(1-\tfrac{\mu_2}{L_2})}\left(\tfrac{1}{L_2}\left\|g_2^i-g_2^j\right\|^2+\mu_2\left\|x^i-x^j\right\|^2-\tfrac{2\mu_2}{L_2}\left\langle g_2^j-g_2^i,x^j-x^i\right\rangle\right)\\
&\leq f_2^i-f_2^j-\left\langle g_2^j, x^i-x^j\right\rangle \ \ i, j\in\left\{1, 2\right\}    \\
\nonumber& \ f_1^k-f_2^k\geq f^\star  \ \ k\in\left\{1, 2\right\}\\
\nonumber & g_2^1=g_1^2\\
\nonumber  & \ \left(f_1^k-f_2^k\right)-f^\star\leq \tfrac{1}{2\eta}\|g_1^k-g_2^k\|^2,  \ \ k\in\left\{1, 2\right\}.
\end{align}
}}
By doing constraint aggregation in problem \eqref{P33} as before (i.e.\ demonstrating a dual feasible solution and using weak duality),
 we obtain the following linear convergence rate for the DCA under the PL inequality.
 \begin{theorem}\label{theorem_PL}
Let $f_1\in\mathcal{F}_{\mu_1, L_1}({\mathbb{R}^n})$ and $f_2\in\mathcal{F}_{\mu_2, L_2}({\mathbb{R}^n})$. If
    $L_1$ or $L_2$ is finite and if $f$ satisfies PL inequality on $X=\{x: f(x)\leq f(x^1)\}$, then for $x^2$  from Algorithm \ref{Alg}, we have
\begin{align}
\frac{f(x^2)-f^\star}{f(x^1)-f^\star}\leq \left(\frac{1-\frac{\eta}{L_1}}{1+\frac{\eta}{L_2}}\right).
\end{align}\label{DCA_linear_bound}
\end{theorem}
\begin{proof}
  Since the given bound is independent of $\mu_1$ and $\mu_2$, without loss of generality,
   we assume that $\mu_1=\mu_2=0$. In addition, we assume that $f^\star=0$. Direct calculation shows that
\begin{align*}
&{\left(f_1^2-f_2^2\right)-f^\star}-\left(\frac{1-\frac{\eta}{L_1}}{1+\frac{\eta}{L_2}}\right)\left( {\left(f_1^1-f_2^1\right)-f^\star} \right)+\left(\frac{1}{1+\frac{\eta}{L_2}}\right)\times \\
&\left( f_1^1-f_1^{2}-\left\langle g_1^{2}, x^1-x^{2}\right\rangle-\tfrac{1}{2L_1}\left\|g_1^1-g_1^{2}\right\|^2\right)\\
&+\left(\frac{1}{1+\frac{\eta}{L_2}}\right)\left( f_2^2-f_2^{1}-\left\langle g_1^{2}, x^2-x^{1}\right\rangle-\tfrac{1}{2L_2}\left\|g_1^2-g_2^{2}\right\|^2\right)+\left(\frac{\frac{\eta}{L_1}}{1+\frac{\eta}{L_2}}\right)\times \\
&\left(\frac{1}{2\eta}\left\|g_1^1-g_1^2\right\|^2-f_1^1+f_2^1\right)+\left(\frac{\frac{\eta}{L_2}}{1+\frac{\eta}{L_2}}\right)\left(\frac{1}{2\eta}\left\|g_1^2-g_2^2\right\|^2-f_1^2+f_2^2\right)=0.
\end{align*}
As all the multipliers in the last expression are non-negative, for any feasible solution of problem \eqref{P3}, we have
$$
f(x^2)-f^\star- \left(\frac{1-\frac{\eta}{L_1}}{1+\frac{\eta}{L_2}}\right)\left(f(x^1)-f^\star  \right)\leq 0,
$$
completing the proof.
\end{proof}
Note that Theorem \ref{le2018convergence} by Le Thi et al. \cite{le2018convergence} does not imply
 Theorem \ref{theorem_PL} if inequality \eqref{Lojasiewicz} holds on $\{x: f(x)\leq f(x^1)\}$
  with $\theta=\tfrac{1}{2}$, since we assume neither strong convexity of $f_1$ or $f_2$, nor boundedness of the sequence of iterates.
   Moreover, we give explicit expressions for the constants that determine the linear convergence rate of the sequence of objective values.

\section{Conclusion}
We have shown that the performance estimation framework of Drori and Teboulle \cite{drori2014performance} yields new insights into the convergence
behavior of the Difference-of-convex algorithm (DCA).
As future work, one may also consider the convergence of the DCA on more restricted classes of DC problems, e.g.\
where $f_1$ and $f_2$ are convex polynomials, as studied in \cite{ahmadi2018dc}.
For constrained problems, even the case where $f_1$ and $f_2$ are quadratic polynomials is of interest, e.g.\ in the study
 of (extended) trust region problems.

\subsection*{Acknowledgement}
This work was supported by the
 Dutch Scientific Council (NWO)  grant \emph{Optimization for and with Machine Learning}, OCENW.GROOT.2019.015.

\end{document}